\documentclass[11pt]{article}
\usepackage{fullpage} 
\usepackage{blkarray}
\usepackage{ucs}
\usepackage{amssymb}

\usepackage{amsmath}
\usepackage{amsthm}
\usepackage{latexsym}
\usepackage[utf8]{inputenc}
\usepackage[english]{babel}
\usepackage{graphicx}
\usepackage{wrapfig}
\usepackage{caption}
\usepackage{subcaption}
\usepackage{txfonts}
\usepackage{lmodern}
\usepackage{mathrsfs}
\usepackage{algorithm}
\usepackage{dsfont}
\usepackage{enumerate}
\usepackage[hidelinks]{hyperref}
\usepackage{multicol}
\usepackage{csquotes}
\usepackage{stmaryrd}
\usepackage{tikz}
\usepackage{comment}
\usepackage{enumitem}
\usepackage{cite}
\usepackage{longtable}

\newtheorem{theo}{Theorem}

\newtheorem{lemma}[theo]{Lemma}

\newtheorem{claim}[theo]{Claim}

\numberwithin{theo}{section}

\newcommand{\conv}{{\rm conv}}
\newcommand{\diam}{{\rm diam}}

 \allowdisplaybreaks

\tikzset{
vtx/.style={inner sep=1.1pt, outer sep=0pt, circle, fill,draw}
}

\usepackage{makecell}

\usepackage{xparse,xpatch,xcolor,tikz}
\usepackage{pgfmath}
\usepackage{xifthen}
\usetikzlibrary{decorations.text, arrows.meta,calc,shadows.blur,shadings}
\usetikzlibrary{shapes.geometric, arrows, positioning}
\usetikzlibrary{quotes,angles}
\usetikzlibrary{calc}

\tikzstyle{vertex}=[circle, draw, inner sep=0pt, fill, minimum size=6pt]
\newcommand{\vertex}{\node[vertex]}
\newcommand{\vc}[1]{\ensuremath{\vcenter{\hbox{#1}}}}

\tikzset{unlabeled_vertex/.style={inner sep=1.7pt, outer sep=0pt, circle, fill}} 
\tikzset{labeled_vertex/.style={inner sep=2.2pt, outer sep=0pt, rectangle, fill=yellow, draw=black}} 
\tikzset{edge_color0/.style={color=black,line width=1.2pt,opacity=0.5}} 
\tikzset{edge_color1/.style={color=red,  line width=1.2pt,opacity=1}} 
\tikzset{edge_color2/.style={color=black, line width=1.2pt,opacity=1}} 

\tikzset{vertex_color1/.style={inner sep=1.7pt, outer sep=0pt, draw, circle, fill=red}} 
\tikzset{vertex_color2/.style={inner sep=1.7pt, outer sep=0pt, draw, circle, fill=blue}} 
\tikzset{vertex_color3/.style={inner sep=1.7pt, outer sep=0pt, draw, circle, fill=green}} 
\tikzset{labeled_vertex_color1/.style={inner sep=2.2pt, outer sep=0pt, draw, rectangle, fill=red}} 
\tikzset{labeled_vertex_color2/.style={inner sep=2.2pt, outer sep=0pt, draw, rectangle, fill=blue}} 
\tikzset{labeled_vertex_color3/.style={inner sep=2.2pt, outer sep=0pt, draw, rectangle, fill=green}} 
\def\outercycle#1#2{ \draw \foreach \x in {0,1,...,#2}{(0.5*\x,0) coordinate(x\x)}; 
\path (0,0.3) -- (1,0.3); 
} 
\def\drawhypervertex#1#2{ \draw[edge_color2] (x#1)++(0,-0.2-0.2*#2)+(-0.2,0) -- +(0.2,0);} 
 
\def\drawhyperedge#1#2{ \draw[dotted] (x0)++(0,-0.2-0.2*#1)--++(0.5*#2-0.5,0);
\path (0,-0.4-0.2*#1) -- (0,0); 
} 

\tikzset{
vtx/.style={inner sep=1.1pt, outer sep=0pt, circle, fill,draw}, 
vtxl/.style={inner sep=1.1pt, outer sep=0pt, rectangle, fill=yellow,draw=black}, 
hyperedge/.style={fill=pink,opacity=0.5,draw=black}, 
}

\begin{document}

\title{Almost Congruent Triangles}

\author{%
J\'ozsef Balogh\footnote{Department of Mathematics, University of Illinois at Urbana-Champaign, Urbana, Illinois 61801, USA. E-mail: \texttt{jobal@illinois.edu}. Research is partially supported by NSF Grant DMS-1764123, NSF RTG grant DMS 1937241, Arnold O. Beckman Research Award (UIUC Campus Research Board RB 22000), and the Langan Scholar Fund (UIUC).} 
\and Felix Christian Clemen \footnote {Department of Mathematics, Karlsruhe Institute of Technology, 76131 Karlsruhe, Germany, E-mail: \texttt{felix.clemen@kit.edu}.}
 \and Adrian Dumitrescu \footnote {Algoresearch L.L.C., Milwaukee, WI 53217, USA, E-mail: \texttt{ad.dumitrescu@algoresearch.org}. }
}

\maketitle
\begin{abstract}
  Almost $50$ years ago Erd\H{o}s and Purdy asked the following question:
  Given $n$ points in the plane, how many triangles can be approximate congruent to equilateral triangles?
  
  They pointed out that  by dividing the points evenly into three small clusters built 
  around the three vertices of a fixed equilateral triangle, one gets at least $\left\lfloor \frac{n}{3} \right\rfloor \cdot \left\lfloor \frac{n+1}{3} \right\rfloor \cdot \left\lfloor \frac{n+2}{3} \right\rfloor$
  such approximate copies.
  In this paper we provide a matching upper bound and thereby answer  their question.

More generally, for every  triangle $T$ we determine the maximum number of approximate congruent triangles to $T$ in a point set of size $n$.

Parts of our proof are based on hypergraph Tur\'an theory: for each point set in the plane and a triangle $T$, we construct a $3$-uniform hypergraph $\mathcal{H}=\mathcal{H}(T)$, which contains no hypergraph as a subgraph  from a family of forbidden hypergraphs $\mathcal{F}=\mathcal{F}(T)$. Our upper bound on  the number of edges  of $\mathcal{H}$ will 
determine the maximum number of triangles that are approximate congruent to $T$. 


\end{abstract}

\section{Introduction}
There is a great variety of extremal problems on various properties of finite point sets in the plane. Here, we study  a problem concerning almost congruent triangles. Let $T$ be an arbitrary triangle with side lengths $a,b,c$. Let $\varepsilon>0$ and define $\varepsilon'=\varepsilon  \cdot \min\{a,b,c\}$. A triangle $T'=A'B'C'$ is $\varepsilon$-\emph{congruent} to $T$, if there exists points $A,B,C$ such that the triangle $ABC$ is congruent to $T$, and $A',B',C'$ are contained in disks of radius $\varepsilon'$ around $A,B$ and $C$, respectively. Given a triangle $T$, $\varepsilon>0$ and an integer $n$, denote by $h(n,T,\varepsilon)$ the maximum number of triangles $\varepsilon$-\emph{congruent}  to $T$ in a point set $P\subseteq \mathbb{R}^2$ of size $n$. Further, let
\begin{align*}
h(n,T):=\min_{\varepsilon>0}h(n,T,\varepsilon).
\end{align*}
Note that if two triangles $T,T'$ are similar to each other,  then $h(n,T)=h(n,T')$, simply by scaling an extremal construction. 

In 1975, Erd\H{o}s and Purdy~\cite{MR0392837} asked to determine $h(n,T)$ for $T$ being an equilateral triangle and provided the following lower bound construction: Distribute $n$ points as evenly as possible into three small circular clusters centered at the vertices of $T$, where the cluster radius is set at $\varepsilon \cdot \min (a,b,c)$. We answer the question of Erd\H{o}s and Purdy by providing a matching upper bound.
\begin{theo}
\label{equilaterl}
Let $T$ be an equilateral triangle. Then for every positive integer $n$, we have
\begin{align*}
h(n,T)=\left\lfloor \frac{n}{3} \right\rfloor \cdot \left\lfloor \frac{n+1}{3} \right\rfloor \cdot \left\lfloor \frac{n+2}{3} \right\rfloor.
\end{align*}
\end{theo}
Furthermore, up to divisibility conditions, we give a complete characterization for all other triangles. We say that a triangle is of type $(\alpha,\beta,\gamma)$ if $\alpha\geq \beta\geq \gamma$ are its interior angles. We measure angles in degrees.

\begin{theo}\label{upperbounds}
Let $T$ be a triangle and $n$ be a positive integer.
\begin{itemize}
\item[(a)] Let $T$ be right angled. Then, $h(n,T)\leq\frac{n^3}{16}$, and if additionally $n$ is divisible by $4$, then $h(n,T)=\frac{n^3}{16}$.

\item[(b)] Let $T$ be of type $(120^\circ,30^\circ,30^\circ)$. Then, $h(n,T)\leq \frac{4}{81}n^3$, and if additionally $n$ is divisible by $9$, then $h(n,T)= \frac{4}{81}n^3$.

\item[(c)]Let $T$ be of type $\left(\frac{4\cdot 180}{7}^\circ,\frac{2\cdot 180}{7}^\circ,\frac{180}{7}^\circ\right)$. Then, $h(n,T)\leq\frac{2}{49}n^3$, and if additionally $n$ is divisible by $7$, then $h(n,T)=\frac{2}{49}n^3$.

\item[(d)] Let $T$ be of type $(108^\circ,36^\circ,36^\circ)$ or $(72^\circ,72^\circ,36^\circ)$. Then, $h(n,T)\leq \frac{n^3}{25}$, and if additionally $n$ is divisible by $5$, then $h(n,T)= \frac{n^3}{25}$.

\item[(e)] Let $T$ be not right angled, and not of type $(120^\circ,30^\circ,30^\circ)$, $\left(\frac{4\cdot 180}{7}^\circ,\frac{2\cdot 180}{7}^\circ,\frac{180}{7}^\circ\right)$, $(108^\circ,36^\circ,36^\circ)$ or $(72^\circ,72^\circ,36^\circ)$. Then, $h(n,T)\leq \frac{n^3}{27}$, and if additionally $n$ is divisible by $3$, then $h(n,T)= \frac{n^3}{27}$.
\end{itemize}
\end{theo}
Our paper is organized as follows. In the rest of this section, we discuss related work, present the lower bound constructions for Theorem~\ref{upperbounds} and a simple proof of Theorem~\ref{equilaterl} for which we use hypergraph Tur\'an theory. In Section~\ref{forbidden}, we establish connections between hypergraph Tur\'an theory and the problem of determining $h(n,T)$. In Section~\ref{lagrang}, we use Lagrangians of hypergraphs to translate the problem of determining $h(n,T)$ to a weighted hypergraph Tur\'an problem on few vertices. Finally, in Section~\ref{final}, we combine our results to complete the proof of
Theorem~\ref{upperbounds}.

\subsection{Related work}
As one would expect, the number of triangles congruent to
a given one is much smaller than the number of almost congruent ones.
Let $C(n)$ denote the maximum number of triples in an $n$-element point set
that induce a triangle congruent to a given triangle $T$.
It is known~\cite{Pa17} that $\Omega(n^{1 + c/\log{\log n}})  \leq C(n) \leq O(n^{4/3})$.
For points in convex position, linear bounds are in effect.
For instance, Pach and Pinchasi~\cite{MR2040883} proved that every set of $n$ points in strictly convex position in the plane has at most
$\left \lfloor 2(n - 1)/3 \right \rfloor$ triples that span equilateral triangles with side length 1, and
this bound is best possible for every $n$. 

Regarding similar triangles, one would expect this number to be larger, and indeed,
there exists a positive constant $c$ such that for any triangle $T$ and any $n \geq 3$,
there is an $n$-element point set in the plane with at least $c n^2$ triples
that induce triangles similar to $T$~\cite[Ch.~6.1]{MR2163782}, \cite{Pa17}.
For equilateral triangles, more precise quadratic bounds were obtained by \'{A}brego and Fern\'{a}ndez-Merchant~\cite{MR1727127},
although the non-matching leading constants still leave room for improvement. 

Allowing approximation beyond similarity yields functions of larger order of magnitude. 
In particular, B\'ar\'any and F\"{u}redi~\cite{MR3953886} have shown that the maximum number of
$\varepsilon$-similar copies of an equilateral triangle (suitably defined in terms of angles)
is $n^3/24 -O(n)$. Balogh, Clemen and Lidick\'y~\cite{MR4404492} extended this result to almost all triangles $T$. However, it remains open to give a full characterization, analogue to Theorem~\ref{upperbounds}, for the maximum number of
$\varepsilon$-similar copies of a triangle $T$.

\subsection{Constructions}
In this section we present the lower bound constructions for Theorem~\ref{upperbounds}, see Figure~\ref{Fig:shapes} for illustrations of the underlying point configurations.
\\

\noindent
(a) Let $T$ be an arbitrary right triangle. Fix the rectangle which contains four copies of $T$. Assume that $n$ is a multiple of $4$. Partition $n$ points into four groups, each of size  $n/4$, where each group of points is placed inside a disk of sufficiently small radius centered at a vertex of the rectangle. 
This construction contains $n^3/16$ triangles $\varepsilon$-congruent to $T$.\\

\noindent
(b) Let $T$ be a triangle of type $(120^\circ,30^\circ,30^\circ)$ with side lengths $a,b,b$. Fix an equilateral triangle with sides of length $a$. Assume that $n$ is divisible by $9$. Partition $n$ points into four groups, three of size $2n/9$ and the last of size $n/3$, where the three smaller groups are placed inside small disks centered at the vertices and the large group inside a small disk centered at the center of the equilateral triangle. This construction contains $4n^3/81$ triangles $\varepsilon$-congruent to $T$.\\

\noindent
(c) Let $T$ be a triangle of type $\left(\frac{4\cdot 180}{7}^\circ,\frac{2\cdot 180}{7}^\circ,\frac{180}{7}^\circ\right)$ with side lengths $a<b<c$. Fix a regular $7$-gon with sides of length $a$. Assume that $n$ is a multiple of $7$. Partition $n$ points into seven groups, each of size  $n/7$, where each group of points is placed inside a disk of sufficiently small radius centered at a vertex of the $7$-gon. 
This construction contains $2n^3/49$ triangles $\varepsilon$-congruent to $T$.\\

\noindent
(d) Let $T$ be a triangle of type $(108^\circ,36^\circ,36^\circ)$ (or ($72^\circ,72^\circ,36^\circ)$) with side lengths $a,b,b$ (or $a,a,b$), where $a>b$. Fix a regular pentagon with sides of length $b$. Assume that $n$ is a multiple of $5$. Partition $n$ points into five groups, each of size  $n/5$, where each group of points is placed inside a disk of sufficiently small radius centered at a vertex of the pentagon.
This construction contains $n^3/25$ triangles $\varepsilon$-congruent to $T$.\\

\noindent
(e) Let $T$ be an arbitrary triangle. Assume that $n$ is a multiple of $3$. Partition $n$ points into three groups, each of size  $n/3$, where each group of points is placed inside a disk of sufficiently small radius centered at a vertex of the triangle. This construction contains $n^3/27$ triangles $\varepsilon$-congruent to $T$.\\

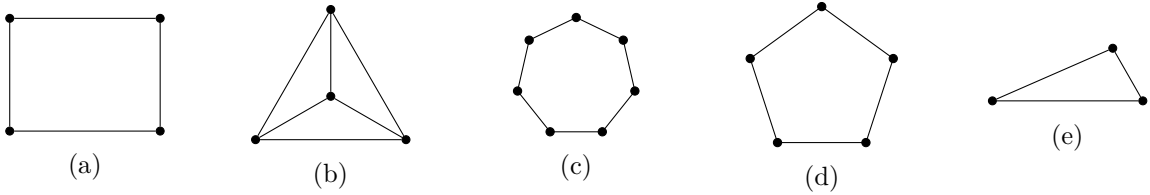
\begin{figure}[h!]
\begin{subfigure}{.19\textwidth}
\centering
    \begin{tikzpicture}[]

\vertex[minimum size = 3pt](0) at (0,0){};
\vertex[minimum size = 3pt](1) at (2,0) {};
\vertex[minimum size = 3pt](2) at (0,1.5) {};
\vertex[minimum size = 3pt](3) at (2,1.5) {};

\draw (0) -- (1);
\draw(0) -- (2);
\draw (3) -- (1);
\draw (3) -- (2);

\end{tikzpicture}
\caption*{(a)}
\end{subfigure}
\begin{subfigure}{.19\textwidth}
\centering
    \begin{tikzpicture}[]

\vertex[minimum size = 3pt](0) at (0,0){};
\vertex[minimum size = 3pt](1) at (2,0) {};
\vertex[minimum size = 3pt](2) at (1,1.732) {};
\vertex[minimum size = 3pt](3) at (1,0.577) {};

\draw (0) -- (1);
\draw(0) -- (2);
\draw (2) -- (1);
\draw (3) -- (0);
\draw (3) -- (1);
\draw (3) -- (2);

\end{tikzpicture}
\caption*{(b)}
\end{subfigure}
\begin{subfigure}{.19\textwidth}
\centering
    \begin{tikzpicture}[]

\draw
\foreach \i in {1,...,7}{
(90+360/7*\i:0.8) coordinate(\i) node[vtx]{}
};

\draw(1) -- (2);
\draw (2) -- (3);
\draw (3) -- (4);
\draw (4) -- (5);
\draw (5) -- (6);
\draw (6) -- (7);
\draw (7) -- (1);

\end{tikzpicture}
\caption*{(c)}
\end{subfigure}
\begin{subfigure}{.19\textwidth}
\centering
    \begin{tikzpicture}[]

\draw
\foreach \i in {1,...,5}{
(90+360/5*\i:1) coordinate(\i) node[vtx]{}
};

\draw(1) -- (2);
\draw (2) -- (3);
\draw (3) -- (4);
\draw (4) -- (5);
\draw (5) -- (1);

\end{tikzpicture}
\caption*{(d)}
\end{subfigure}
\begin{subfigure}{.19\textwidth}
\centering
    \begin{tikzpicture}[]

\vertex[minimum size = 3pt](0) at (0,0){};
\vertex[minimum size = 3pt](1) at (2,0) {};
\vertex[minimum size = 3pt](2) at (1.6,0.7) {};

\draw(0) -- (1);
\draw (0) -- (2);
\draw (1) -- (2);
\end{tikzpicture}
\caption*{(e)}
\end{subfigure}

\caption{Point configurations for the lower bound constructions in Theorem~\ref{upperbounds}. }
\label{Fig:shapes}
\end{figure}

\subsection{A warm-up: Proof of Theorem~\ref{equilaterl} }

For our proofs we will use hypergraph Tur\'an theory for $3$-uniform hypergraphs, named here $3$-graphs. For a good overview on this topic, see the excellent survey of Keevash~\cite{Keevashsurvey}. Let $F$ and $G$ be $3$-graphs. We say that $G$ is $F$-\emph{free} if it does not have a 
subhypergraph isomorphic to $F$. Given a family of $3$-graphs $\mathcal{F}$, we say that $G$ is $\mathcal{F}$-\emph{free} if $G$ is $F$-free for every $F\in\mathcal{F}$. For a positive integer $n$, the \emph{Tur\'an function} of $\mathcal{F}$, denoted by $\textup{ex}(n,\mathcal{F})$, is defined to be the maximum number of edges in an $\mathcal{F}$-free $n$-vertex $3$-graph. The \emph{Tur\'an density} of $\mathcal{F}$ is 
\begin{align*}
\pi(\mathcal{F})=\lim_{n\to \infty} \frac{\textup{ex}(n,\mathcal{F})}{\binom{n}{3}}.
\end{align*}
If $\mathcal{F}=\{H\}$ for a single $3$-graph $H$, we simply write $\textup{ex}(n,H)$ and $\pi(H)$ instead of $\textup{ex}(n,\{H\})$ and $\pi(\{H\})$ respectively. The \emph{link graph} of a $3$-graph $H$ is the graph with vertex set $V(H)$, where a pair of vertices is an edge iff it is contained in some edge in $H$.

Given a triangle $T$ and a finite point set $P\subseteq \mathbb{R}^2$, we denote by $\mathcal{H}(T,P)$ the $3$-graph with vertex set $P$ and edges being those triples forming triangles that are congruent to $T$. Given a triangle $T$, a finite point set $P\subseteq \mathbb{R}^2$ and $\varepsilon>0$, we denote by $\mathcal{H}(T,P,\varepsilon)$ the $3$-graph with vertex set $P$ and edges being those triples forming triangles that are $\varepsilon$-congruent to $T$. We say that a hypergraph $H$ is \emph{forbidden} for a triangle $T$, if there exists $\varepsilon=\varepsilon(T,H)>0$ such that for every point set $P$ of size $|V(H)|$, the $3$-graph $\mathcal{H}(T,P,\varepsilon)$ is $H$-free.

A strategy for our proofs is to find forbidden $3$-graphs for different triangles and then use the bounds for the corresponding hypergraph Tur\'an problem. As a warm-up for the proof of our main result, we present a short proof of Theorem~\ref{equilaterl}.
\begin{proof}
Denote by $T$ the equilateral triangle of unit side-length and let $$s(n):= \left\lfloor\frac{n}{3} \right\rfloor \cdot \left\lfloor\frac{n+1}{3} \right\rfloor\cdot \left\lfloor \frac{n+2}{3} \right\rfloor.$$  A $3$-graph $G$ is  \emph{cancellative}, see~\cite{Bollobascancellative}, if no symmetric difference of two edges of $G$ is  contained in a third edge of $G$. Note that being cancellative is equivalent to $G$ being $\{F_5,K_4^{3-}\}$-free, where $F_5$ is the $3$-graph with vertex set $V(F_5)=[5]$ and edges $123,124,345$, and $K_4^{3-}$ is the $3$-graph with vertex set $[4]$ and edges $123,124,134$, which in the literature is sometimes also called $F_4$.

A classical result of Bollob\'{a}s~\cite{Bollobascancellative} states that the maximum number of edges in a cancellative $3$-graph is $s(n)$, and equality is obtained by the complete balanced $3$-partite $3$-graph. 
With other words, Bollob\'{a}s proved that $\textup{ex}(n,\{F_5,K_4^{3-}\})=s(n)$. Now, let $\varepsilon>0$ be sufficiently small and let $P$ be a point set of size $n$. 
Since the number of triangles that are $\varepsilon$-congruent to $T$ in $P$ equals the number of edges in $\mathcal{H}(T,P,\varepsilon)$, it suffices to show that $F_5$ and $K_4^{3-}$ are forbidden for $T$.

If $F_5$ was not forbidden for $T$, then there would exist five distinct points $p_1,p_2,p_3,p_4,p_5\in\mathbb{R}^2$ such that the triangles $p_1p_2p_3,p_1p_2p_4,p_3p_4p_5$ are $\varepsilon$-congruent to $T$. Then, by the triangle-inequality, the distance between each pair of the four points $p_1,p_2,p_3,p_4$ is between $1-2\varepsilon$ and $1+2\varepsilon$, however this is not possible  for $\varepsilon$ sufficiently small by the following lemma.

\begin{lemma} 
\label{4points}
  Let $a$ be a positive real number and $P$ be a set of $4$ points in the plane where the minimum pairwise distance is at least $a$. Then $\diam(P) \geq \sqrt2 a$. 
\end{lemma}
\begin{proof}
 We distinguish two cases, when $P$ is in convex position, and when it is not.
  In the former case, at least one of the four angles of the quadrilateral is at least $90^\circ$,
  and then by the Cosine Law, the corresponding diagonal that subtends this angles is
  of length at least $\sqrt{a^2 + a^2} =\sqrt2 a$, as required. In the latter case, $\conv(P)$ is a triangle and at least one of the three angles at the interior point
  is at least $120^\circ$. By the Cosine Law, the corresponding triangle side that subtends this angle is
  of length at least $\sqrt{a^2 + a^2 + 2 a^2\cdot \cos{120^\circ}} =\sqrt3 a > \sqrt2 a$, as required.
\end{proof}
If $K_4^{3-}$ was not forbidden for $T$, then  there would exist four distinct points $p_1,p_2,p_3,p_4\in\mathbb{R}^2$ such that the triangles $p_1p_2p_3,p_1p_2p_4,p_1p_3p_4$ are $\varepsilon$-congruent to $T$. Then, the distance between each pair of the four points $p_1,p_2,p_3,p_4$ is between $1-2\varepsilon$ and $1+2\varepsilon$. Again, this is not possible by Lemma~\ref{4points}  for $\varepsilon$ sufficiently small. We conclude $h(n,T)\leq s(n)$. 
\end{proof}
Note that, in fact, every 3-graph $H$, which contains four vertices $v_1,v_2,v_3,v_4\in V(H)$ such that every pair $v_iv_j$ for $1\leq i<j\leq 4$ is contained in some edge in $H$, is forbidden for $T$.
We remark that Frankl and F\"uredi~\cite{F5Frankl} proved that $\textup{ex}(n,F_5)=s(n)$ for $n\geq 3000$, which was extended by Keevash and Mubayi~\cite{MubayiF5} to hold for $n\geq 33$. Thus, for those values of $n$, it would have been sufficient just to use that $F_5$ is forbidden for $T$. 

\section{Forbidden Hypergraphs}
\label{forbidden}
The following hypergraphs will appear as forbidden hypergraphs for some triangles.

\begin{center}
  \setlength{\tabcolsep}{1pt}
  \renewcommand{\arraystretch}{2.0}
  \begin{longtable}{ | l|| c |  c | c|  }
    \hline
    \ $H$ &   Edges   &  Visualization & Name  / Description  \\
    \hline\hline
    \ $K_4^3$ 
    &$123, 124, 134, 234 $
    & 
\vc{\begin{tikzpicture}\outercycle{5}{4}
\draw (x0) node[unlabeled_vertex]{};\draw (x1) node[unlabeled_vertex]{};\draw (x2) node[unlabeled_vertex]{};\draw (x3) node[unlabeled_vertex]{};
\drawhyperedge{0}{4}
\drawhypervertex{0}{0}
\drawhypervertex{1}{0}
\drawhypervertex{2}{0}
\drawhyperedge{1}{4}
\drawhypervertex{0}{1}
\drawhypervertex{1}{1}
\drawhypervertex{3}{1}
\drawhyperedge{2}{4}
\drawhypervertex{0}{2}
\drawhypervertex{2}{2}
\drawhypervertex{3}{2}
\drawhyperedge{3}{4}
\drawhypervertex{1}{3}
\drawhypervertex{2}{3}
\drawhypervertex{3}{3}
\end{tikzpicture} 
}    & Complete $3$-graph on four vertices
   \\
\hline
    \ $K_4^{3-}$ 
    &$123, 124, 134$
    &
\vc{\begin{tikzpicture}\outercycle{5}{4}
\draw (x0) node[unlabeled_vertex]{};\draw (x1) node[unlabeled_vertex]{};\draw (x2) node[unlabeled_vertex]{};\draw (x3) node[unlabeled_vertex]{};
\drawhyperedge{0}{4}
\drawhypervertex{0}{0}
\drawhypervertex{1}{0}
\drawhypervertex{2}{0}
\drawhyperedge{1}{4}
\drawhypervertex{0}{1}
\drawhypervertex{1}{1}
\drawhypervertex{3}{1}
\drawhyperedge{2}{4}
\drawhypervertex{0}{2}
\drawhypervertex{2}{2}
\drawhypervertex{3}{2}
\path (0,0.4) -- (0,-1.1); 
\end{tikzpicture} 
}   
& \makecell{\ Complete $3$-graph minus one edge  \\ on 4 vertices  \\ }

   \\
   \hline
\ $F_{3,2}$ & $123, 145, 245, 345$
&
\vc{\begin{tikzpicture}\outercycle{6}{5}
\draw (x0) node[unlabeled_vertex]{};\draw (x1) node[unlabeled_vertex]{};\draw (x2) node[unlabeled_vertex]{};\draw (x3) node[unlabeled_vertex]{};\draw (x4) node[unlabeled_vertex]{};
\drawhyperedge{0}{5}
\drawhypervertex{0}{0}
\drawhypervertex{1}{0}
\drawhypervertex{2}{0}
\drawhyperedge{1}{5}
\drawhypervertex{0}{1}
\drawhypervertex{3}{1}
\drawhypervertex{4}{1}
\drawhyperedge{2}{5}
\drawhypervertex{1}{2}
\drawhypervertex{3}{2}
\drawhypervertex{4}{2}
\drawhyperedge{3}{5}
\drawhypervertex{2}{3}
\drawhypervertex{3}{3}
\drawhypervertex{4}{3}
\end{tikzpicture} 
}
& \makecell{\ $F_{3,2}$-free $3$-graphs are called $3$-graphs \ \\ with independent  neighborhood}
   \\
   \hline
   \ $J_4$
   &
   $123, 124, 125, 134, 135, 145$
   &
\vc{\begin{tikzpicture}\outercycle{6}{5}
\draw (x0) node[unlabeled_vertex]{};\draw (x1) node[unlabeled_vertex]{};\draw (x2) node[unlabeled_vertex]{};\draw (x3) node[unlabeled_vertex]{};\draw (x4) node[unlabeled_vertex]{};
\drawhyperedge{0}{5}
\drawhypervertex{0}{0}
\drawhypervertex{1}{0}
\drawhypervertex{2}{0}
\drawhyperedge{1}{5}
\drawhypervertex{0}{1}
\drawhypervertex{1}{1}
\drawhypervertex{3}{1}
\drawhyperedge{2}{5}
\drawhypervertex{0}{2}
\drawhypervertex{1}{2}
\drawhypervertex{4}{2}
\drawhyperedge{3}{5}
\drawhypervertex{0}{3}
\drawhypervertex{2}{3}
\drawhypervertex{3}{3}
\drawhyperedge{4}{5}
\drawhypervertex{0}{4}
\drawhypervertex{2}{4}
\drawhypervertex{4}{4}
\drawhyperedge{5}{5}
\drawhypervertex{0}{5}
\drawhypervertex{3}{5}
\drawhypervertex{4}{5}
\end{tikzpicture} 
}   
&  \makecell{$5$-vertex $3$-graph with all triples \\ containing one  fixed vertex}
   \\ 
   \hline
\ $F_5$
&
$123, 124, 345$
&
\vc{\begin{tikzpicture}\outercycle{6}{5}
\draw (x0) node[unlabeled_vertex]{};\draw (x1) node[unlabeled_vertex]{};\draw (x2) node[unlabeled_vertex]{};\draw (x3) node[unlabeled_vertex]{};\draw (x4) node[unlabeled_vertex]{};
\drawhyperedge{0}{5}
\drawhypervertex{0}{0}
\drawhypervertex{1}{0}
\drawhypervertex{2}{0}
\drawhyperedge{1}{5}
\drawhypervertex{0}{1}
\drawhypervertex{1}{1}
\drawhypervertex{3}{1}
\drawhyperedge{2}{5}
\drawhypervertex{2}{2}
\drawhypervertex{3}{2}
\drawhypervertex{4}{2}
\end{tikzpicture} 
}
& Generalized triangle
   \\
   \hline
\ $C_5$
&
$123, 234, 345, 145, 125$
&
\vc{\begin{tikzpicture}\outercycle{6}{5}
\draw (x0) node[unlabeled_vertex]{};\draw (x1) node[unlabeled_vertex]{};\draw (x2) node[unlabeled_vertex]{};\draw (x3) node[unlabeled_vertex]{};\draw (x4) node[unlabeled_vertex]{};
\drawhyperedge{0}{5}
\drawhypervertex{0}{0}
\drawhypervertex{1}{0}
\drawhypervertex{2}{0}
\drawhyperedge{1}{5}
\drawhypervertex{1}{1}
\drawhypervertex{2}{1}
\drawhypervertex{3}{1}
\drawhyperedge{2}{5}
\drawhypervertex{2}{2}
\drawhypervertex{3}{2}
\drawhypervertex{4}{2}
\drawhyperedge{3}{5}
\drawhypervertex{0}{3}
\drawhypervertex{3}{3}
\drawhypervertex{4}{3}
\drawhyperedge{4}{5}
\drawhypervertex{0}{4}
\drawhypervertex{1}{4}
\drawhypervertex{4}{4}
\end{tikzpicture} 
} & Tight cycle of length 5
   \\
   \hline  
    \caption{\label{tab:pic}Description of some $3$-graphs. 
    }
    \label{listofforb}
  \end{longtable}
\end{center}
\vspace{-0.5cm}
For an overview of bounds on Tur\'an densities of these and other hypergraphs, see \cite{MR4421399}.

We say that a hypergraph $H$ with vertex set $[k]$ is \emph{exactly forbidden} for a triangle $T$ iff there do not exist $k$ points $p_1,\ldots,p_k\in \mathbb{R}^2$ (not necessarily distinct) such that for every edge $xyz\in E(H)$, the triangle $p_xp_yp_z$ is congruent to $T$. We call a $3$-graph $H$ on $k$ vertices $\emph{dense}$ if there exists a vertex ordering $v_1,v_2,\ldots, v_{k}$ such that for every $3\leq i \leq k $ there exists at least one edge $e_i\in E(H[\{v_1,\ldots,v_i\}])$ containing $v_i$. Note that all hypergraphs in Table~\ref{listofforb} are dense; indeed for $F_{3,2}$ consider the vertex ordering $4,5,1,2,3$, and for all other hypergraphs simply consider the canonical vertex ordering.

Given a point $p\in \mathbb{R}^2$ and some $\varepsilon>0$, we denote by $B_\varepsilon(p)$ the closed ball (disk) of radius $\varepsilon$ centered at $p$, i.e., $B_\varepsilon(p)$ is the set of points $p'\in \mathbb{R}^2$ such that $|pp'|\leq \varepsilon$. Recall that a hypergraph $H$ is \emph{forbidden} for a triangle $T$, if there exists an $\varepsilon=\varepsilon(T,H)>0$ such that for every point set $P$ of size $|V(H)|$, the $3$-graph $\mathcal{H}(T,P,\varepsilon)$ is $H$-free. The following lemma shows that the two different notions of forbidden hypergraphs are essentially the same. Later, it will more convenient to work with the definition of exactly forbidden hypergraphs.
\begin{lemma}
\label{exactforbidden}
Let $H$ be a dense hypergraph and $T$ a triangle. 
A hypergraph $H$ is forbidden for $T$ iff $H$ is exactly forbidden for $T$.
\end{lemma}
\begin{proof}
Let $H$ be a dense $3$-graph on vertex set $[k]$ and let $T$ be a triangle with side lengths $a,b,c$.

If $H$ is not exactly forbidden for $T$, then there exists $p_1,\ldots,p_k\in \mathbb{R}^2$ (not necessarily distinct) such that for every edge $xyz\in E(H)$, the triangle $p_xp_yp_z$ is congruent to $T$. Let $\varepsilon>0$ and set $\varepsilon'=\varepsilon\cdot \min\{a,b,c\}$. Then we can choose $p_i'\in B_{\varepsilon'}(p_i)$ for all $i\in[k]$ such that all points $p_1',\ldots, p_k'\in \mathbb{R}^2$ are distinct. Now, if $xyz\in E(H)$, then the triangle $p_x'p_y'p_z'$ is $\varepsilon$-congruent to $T$. Thus, $H$ is not forbidden for $T$.

For the other direction, assume that $H$ is not forbidden for $T$.  
For $0\leq i \leq k$, a sequence $(v^n)_n$ of vectors of $k$ points $v^n=(v^n_1,\ldots,v^n_k)\in (\mathbb{R}^2)^k$ is called $i$-\emph{partial}, if there exists a sequence of real numbers $\varepsilon_n \downarrow 0$ such that $v^n_xv^n_yv^n_z$ is $\varepsilon_n$-congruent to $T$ for $xyz\in E(H)$ and all $n\in  \mathbb{N}$, and further $v^n_j=v^{1}_j$ for all $n\in  \mathbb{N}$ and $j\leq i$.  

\begin{claim}
There exists a $k$-partial sequence. 
\end{claim}
\begin{proof}
We construct the $k$-partial sequence iteratively.  
Because $H$ is not forbidden for $T$, for every $\varepsilon>0$ there exists points $p_1,\ldots,p_k\in \mathbb{R}^2$ such that $p_{x}p_{y}p_{z}$ forms a triangle that is $\varepsilon$-congruent to $T$ if $xyz\in E(H)$. In particular, there exists a $0$-partial sequence. After shifting, we can assume that the first point in each vector is the origin, and thus we get a $1$-partial sequence. 

Because $H$ is dense, $123\in E(H)$. Since $123\in E(H)$ and after rotating around the origin, we can assume that the second point $v_n^2$ in each vector is on the $x$-axis with positive $x$-coordinate and has distance at most $2\varepsilon_n \cdot \min\{a,b,c\}$ from one of the points $(a,0),(b,0),(c,0)$. After replacing the second point in each vector with a point from $(a,0),(b,0),(c,0)$ which it is the closest to, we get a 1-partial sequence, where the second point takes one out of three values. 
Going over to a subsequence where the second point is constant, we obtain a $2$-partial sequence. 

Now, assume that for some $2\leq i \leq k-1$ there exists an $i$-partial sequence $(v^n)_n$. We will construct an $(i+1)$-partial sequence from it. Since $H$ is dense, there exists an edge $e_{i+1}=xy(i+1)\in E(H)$ where $x,y\in [i]$. Observe that the distance $|v_x^1v_y^1|$ is $a,b$ or $c$, because $v_x^1v_y^1v_{i+1}^n$ is $\varepsilon_n$-congruent for every $n$ and $\varepsilon_n \downarrow 0$. There are at most $4$ points $p\in \mathbb{R}^2$ such that $v^1_xv^1_yp$ is congruent to $T$. After replacing each $v^n_{i+1}$ with a point among those $4$ points it is the closest to, and going over to a subsequence where the second point is constant, we obtain an $(i+1)$-partial sequence. 
\end{proof}
Let $(v^n)_n$ be a $k$-partial sequence. By definition, $v_j^n=v_j^1$ for all $n\in \mathbb{N}$ and $j\in [k]$. The points $v_1^1,\ldots,v_k^1\in \mathbb{R}^2$ have the property that if $xyz\in E(H)$, then $v_x^1v_y^1v_z^1$ is $\varepsilon_n$-congruent to $T$ for all $n$ for some $\varepsilon_n\downarrow 0$, and thus also congruent to $T$. We conclude that $H$ is not exactly forbidden for $T$. 
\end{proof}

\subsection{Elementary geometry}
In this subsection we collect facts and results from elementary Euclidean geometry. For $P \subseteq \mathbb{R}^2$, let $A(P) := \{|xy|: x,y\in P,  x\neq y\}$. The set $P$ is  an $s$-\emph{distance set} if $|A(P)| = s$. An $s$-\emph{distance set} $P$ is  \emph{maximal}, if it is maximal with respect to inclusion. Two subsets in $\mathbb{R}^2$ are \emph{isomorphic} if there exists a similarity transformation from one to the other.
\begin{lemma}\label{1distance}
A 1-distance set has size at most $3$. 
\end{lemma}
\begin{lemma}[Erd\H{o}s, Kelly~\cite{MR1526679}]
\label{2distances}
A 2-distance set has size at most $5$. Further, if $P$ is a $2$-distance set of size $5$, then $P$ is a regular pentagon.
\end{lemma}

\begin{theo}[Shinohara~\cite{MR2083454}]
\label{3distances}
A 3-distance set has size at most $7$. If $P$ is a 3-distance set of size $7$, then $P$ is a regular $7$-gon, or a regular $6$-gon together with its center, see Figure~\ref{Fig:shapes2}. 
\end{theo}
\vspace{-0.4cm}
\begin{figure}[h!]
\begin{subfigure}{.5\textwidth}
\centering
    \begin{tikzpicture}[]

\draw
\foreach \i in {1,...,7}{
(90+360/7*\i:0.8) coordinate(\i) node[vtx]{}
};

\draw(1) -- (2);
\draw (2) -- (3);
\draw (3) -- (4);
\draw (4) -- (5);
\draw (5) -- (6);
\draw (6) -- (7);
\draw (7) -- (1);

\end{tikzpicture}

\end{subfigure}
\begin{subfigure}{.5\textwidth}
\centering
    \begin{tikzpicture}[]

\draw
\foreach \i in {1,...,6}{
(90+360/6*\i:0.8) coordinate(\i) node[vtx]{}
};

\vertex[minimum size = 3pt](0) at (0,0){};

\draw(1) -- (2);
\draw (2) -- (3);
\draw (3) -- (4);
\draw (4) -- (5);
\draw (5) -- (6);
\draw (6) -- (7);
\draw (7) -- (1);

\end{tikzpicture}

\end{subfigure}

\caption{$3$-distance sets of size $7$.}
\label{Fig:shapes2}
\end{figure}
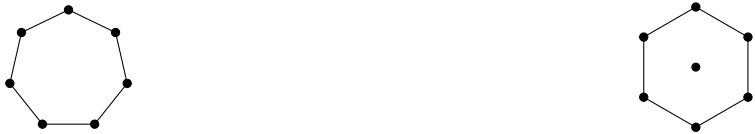
\begin{theo}[Shinohara~\cite{MR2083454}]
\label{3distances6}
Let $P$ be a maximal $3$-distance set of size $6$ such that $A(P)=\{1,b,c\}$, where $1<b< c$. Set $\gamma :=\sqrt{2+\sqrt{3}}$ and  $\tau:= \frac{1+\sqrt{5}}{2}$. Then $$(b,c)\in \left\{(\sqrt{3},2),(2\sin \frac{\pi}{5},2\tau \sin \frac{\pi}{5}),(\sqrt{2},\gamma),(\gamma,\sqrt{2}\gamma)\right\}.$$
\begin{itemize}
\item If $(b,c)=(\sqrt{3},2)$, then $P$ forms an equilateral triangle of side length 2 together with the midpoints of the three sides (see Figure~\ref{Fig:shapes3} (a)).

\item If $(b,c)=(2\sin \frac{\pi}{5},2\tau \sin \frac{\pi}{5})$, then $P$ is a regular 5-gon together with its  center, see Figure~\ref{Fig:shapes3} (b).

\item If $(b,c)=(\gamma,\sqrt{2}\gamma)$, then $P$ is isomorphic to one of the point sets in Figure~\ref{Fig:shapes3} (c) and (d)).

\item If $(b,c)=(\sqrt{2},\gamma)$, then $P$ is isomorphic to one of the point sets in Figure~\ref{Fig:shapes3} (e) and (f)).

\end{itemize}
\end{theo}

\begin{figure}[h!]
\begin{subfigure}{.33\textwidth}
\centering
    \begin{tikzpicture}[]

\vertex[minimum size = 3pt](0) at (-1,0){};
\vertex[minimum size = 3pt](1) at (1,0){};
\vertex[minimum size = 3pt](2) at (0,0){};
\vertex[minimum size = 3pt](3) at (0,1.732){};
\vertex[minimum size = 3pt](4) at (0.5,0.866){};
\vertex[minimum size = 3pt](5) at (-0.5,0.866){};

\draw(0) -- (1);
\draw (1) -- (3);
\draw (0) -- (3);
\draw(4) -- (5);
\draw (4) -- (2);
\draw (5) -- (2);

\draw[red](1) -- (5);
\draw[red](4) -- (0);
\draw[red](3) -- (2);

\vertex[rectangle, minimum size = 4pt] at (0,0){};
\vertex[rectangle, minimum size = 4pt] at (0.5,0.866){};

\end{tikzpicture}
\caption*{(a)}
\end{subfigure}
\begin{subfigure}{.33\textwidth}
\centering
    \begin{tikzpicture}[]

\draw
\foreach \i in {1,...,5}{
(90+360/5*\i:1) coordinate(\i) node[vtx]{}
};

\vertex[minimum size = 3pt](0) at (0,0){};

\draw[red](1) -- (2);
\draw[red] (2) -- (3);
\draw[red] (3) -- (4);
\draw[red] (4) -- (5);
\draw[red] (5) -- (1);
\draw (0) -- (1);
\draw (0) -- (2);
\draw (0) -- (3);
\draw (0) -- (4);
\draw (0) -- (5);

\vertex[rectangle, minimum size = 4pt] at (0,1){};
\vertex[rectangle, minimum size = 4pt] at (0,0){};

\end{tikzpicture}
\caption*{(b)}
\end{subfigure}
\begin{subfigure}{.33\textwidth}
\centering
    \begin{tikzpicture}[]

\vertex[rectangle, minimum size = 4pt](0) at (-0.5,0){};
\vertex[rectangle, minimum size = 4pt](1) at (0.5,0){};
\vertex[minimum size = 3pt](2) at (0, 0.8660){};
\vertex[minimum size = 3pt](3) at (0, 1.8660){};
\vertex[minimum size = 3pt](4) at (-1.366,-0.5){};
\vertex[minimum size = 3pt](5) at (1.366,-0.5){};

\draw(1) -- (2);
\draw (2) -- (0);
\draw (0) -- (1);
\draw (0) -- (4);
\draw (1) -- (5);
\draw (2) -- (3);

\draw[blue] (4) -- (5);
\draw[blue] (3) -- (5);
\draw[blue] (3) -- (4);

\end{tikzpicture}
\caption*{(c)}
\end{subfigure}
\newline

\begin{subfigure}{.33\textwidth}
\centering
    \begin{tikzpicture}[]

\vertex[rectangle, minimum size = 4pt](0) at (-0.5,0){};
\vertex[rectangle,  minimum size = 4pt](1) at (0.5,0){};
\vertex[minimum size = 3pt](2) at (0, 0.8660){};
\vertex[minimum size = 3pt](3) at (0, -1.866){};
\vertex[minimum size = 3pt](4) at (-1.366,-0.5){};
\vertex[minimum size = 3pt](5) at (1.366,-0.5){};

\draw(1) -- (2);
\draw (2) -- (0);
\draw (0) -- (1);
\draw (0) -- (4);
\draw (1) -- (5);

\draw[red] (3) -- (0);
\draw[red] (3) -- (1);
\draw[red] (3) -- (4);
\draw[red] (3) -- (5);
\draw[red] (2) -- (5);
\draw[red] (2) -- (4);

\draw[red] (0) -- (5);
\draw[red] (1) -- (4);

\end{tikzpicture}
\caption*{(d)}
\end{subfigure}
\begin{subfigure}{.33\textwidth}
\centering
    \begin{tikzpicture}[scale=1.2]

\vertex[rectangle, minimum size = 4pt](0) at (-1,0){};
\vertex[minimum size = 3pt](1) at (1,0){};
\vertex[rectangle, minimum size = 4pt](2) at (0, 1.732){};
\vertex[minimum size = 3pt](3) at (0, -0.268){};
\vertex[minimum size = 3pt](4) at (-0.732,1){};
\vertex[minimum size = 3pt](5) at (0.732,1){};

\draw[blue] (1) -- (2);
\draw[blue] (2) -- (0);
\draw[blue] (0) -- (1);
\draw[blue] (1) -- (4);
\draw[blue] (2) -- (3);
\draw[blue] (0) -- (5);

\draw(0) -- (3);
\draw (0) -- (4);
\draw (1) -- (3);
\draw (1) -- (5);
\draw (2) -- (4);
\draw (2) -- (5);

\end{tikzpicture}
\caption*{(e)}
\end{subfigure}
\begin{subfigure}{.33\textwidth}
\centering
    \begin{tikzpicture}[scale=1.4]

\vertex[minimum size = 3pt](0) at (-1,0){};
\vertex[minimum size = 3pt](1) at (1,0){};
\vertex[rectangle, minimum size = 4pt](2) at (0, 1.732){};
\vertex[minimum size = 3pt](3) at (0, 0.268){};
\vertex[minimum size = 3pt](4) at (-0.732,1){};
\vertex[rectangle, minimum size = 4pt](5) at (0.732,1){};

\draw[blue] (1) -- (2);
\draw[blue] (2) -- (0);
\draw[blue] (0) -- (1);
\draw[blue] (1) -- (4);

\draw[blue] (0) -- (5);

\draw(0) -- (3);
\draw (0) -- (4);
\draw (1) -- (3);
\draw (1) -- (5);
\draw (2) -- (4);
\draw (2) -- (5);
\draw(4) -- (3);
\draw(5) -- (3);

\end{tikzpicture}
\caption*{(f)}
\end{subfigure}

\caption{Maximal 3-distance sets of size 6.}
\medskip
\small
Edges with length 1 are colored black, edges with length $b$ are colored red and edges with length $c$ are colored blue. For aesthetic reasons one of the colors is omitted in each of the drawings. In each of the drawings we marked two vertices which are not contained in a triangle with side lengths $1,b,c$. 
\label{Fig:shapes3}
\end{figure}
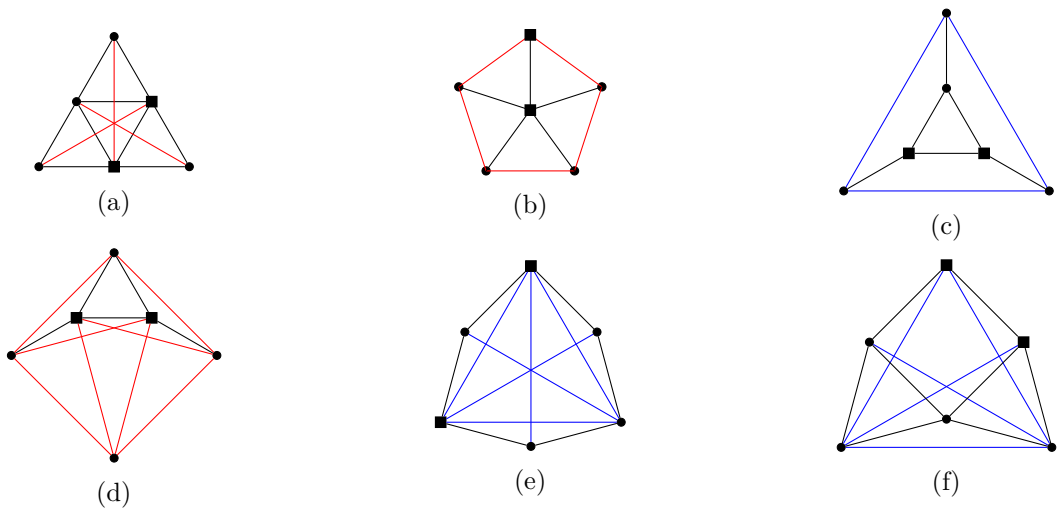

\subsection{Determining forbidden hypergraphs}

\begin{lemma}\label{f32}
Let $T$ be a triangle which is not of type $(90^\circ,60^\circ,30^\circ)$, then $F_{3,2}$ is forbidden for $T$. 
\end{lemma}
\begin{proof}
Recall that $E(F_{3,2})=\{123,145,245,345\}$.
Let $T$ be a triangle with side lengths $a,b,c$ (not necessarily distinct) which is not of type $(90^\circ,60^\circ,30^\circ)$. By Lemma~\ref{exactforbidden} it suffices to show that $F_{3,2}$ is exactly forbidden for $T$. Assume, towards contradiction, that it is not the case. Then there exist points $p_1,p_2,p_3,p_4,p_5\in \mathbb{R}^2$ (not necessarily distinct)
 such that for every edge $xyz\in E(F_{3,2})$, the triangle $p_xp_yp_z$ is congruent to $T$. Since every pair of vertices in $F_{3,2}$ is contained in some edge, the five points $p_1,\ldots,p_5$ are distinct. By Lemma~\ref{1distance} it is not possible that $a=b=c$. If two of the side lengths are equal, then the five points form a regular pentagon by Lemma~\ref{2distances}. For any triangle $T$, the $3$-graph formed by taking edges when the corresponding points form a congruent triangle to $T$ is either isomorphic to $C_5$ or the empty graph. Neither contains a copy of $F_{3,2}$, a contradiction. We conclude that the side lengths $a,b,c$ are pairwise distinct.

{\bf Case:} $T$ is a right triangle.\\
Assume that the right angle is opposite of the side with length $a$. Then $a>b$ and $a>c$.

First, we assume $|p_4p_5|=c$ (the case that $|p_4p_5|=b$ is symmetric). There are four points $a_1,a_2,a_3,a_4\in\mathbb{R}^2$ in the plane which together with $p_4p_5$ form a triangle congruent to $T$. Those four points form a rectangle with side lengths $c$ and $2b$, see Figure~\ref{Fig:F32configa} for an illustration. Therefore, when choosing three out of those four points, the corresponding triangles are not congruent to $T$. Since $145,245,345$ are edges in $F_{3,2}$, we have that $p_1,p_2,p_3\in \{a_1,a_2,a_3,a_4\}$. This contradicts that $123\in E(F_{3,2})$.

Next, assume that $|p_4p_5|=a$. Since $a,b,c$ are pairwise distinct, there are four points $a_1,a_2,a_3,a_4$ in the plane which together with $p_4p_5$ form a triangle congruent to $T$. They form a rectangle with two opposite sides parallel to $p_4p_5$. Since $145,245,345$ are edges in $F_{3,2}$, we have that $p_1,p_2,p_3\in \{a_1,a_2,a_3,a_4\}$. Since $123\in E(F_{3,2})$, the side lengths of the rectangle are $b$ and $c$. After possibly renaming the four points, we have that $a_1a_2$ is parallel to $p_4p_5$ and $a_3a_4$, and $|a_1p_4|<|a_1p_5|$. See Figure~\ref{Fig:F32configb} for an illustration.  
\begin{figure}[h!]
\begin{subfigure}{.5\textwidth}
\centering
    \begin{tikzpicture}
\filldraw (2,0) circle (2pt);
\filldraw (0,0) circle (2pt);
\filldraw (0,1) circle (2pt);
\filldraw (2,1) circle (2pt);
\filldraw (0,-1) circle (2pt);
\filldraw (2,-1) circle (2pt);
\draw (0,0) -- (2,0);
\node[] at (-0.4,0) {$p_4$};
\node[] at (2.5,0) {$p_5$};
\node[] at (-0.4,1) {$a_1$};
\node[] at (2.5,1) {$a_2$};
\node[] at (-0.4,-1) {$a_4$};
\node[] at (2.5,-1) {$a_3$};
\end{tikzpicture}
\caption{\ $|p_4p_5|=c$.}
\label{Fig:F32configa}
\end{subfigure}
\begin{subfigure}{.5\textwidth}
\centering
    \begin{tikzpicture}
\filldraw (2,0) circle (2pt);
\filldraw (0,0) circle (2pt);
\filldraw (0.5,1) circle (2pt);
\filldraw (1.5,1) circle (2pt);
\filldraw (0.5,-1) circle (2pt);
\filldraw (1.5,-1) circle (2pt);
\draw (0,0) -- (2,0);
\node[] at (-0.4,0) {$p_4$};
\node[] at (2.5,0) {$p_5$};
\node[] at (0.1,1) {$a_1$};
\node[] at (1.9,1) {$a_2$};
\node[] at (0.1,-1) {$a_4$};
\node[] at (1.9,-1) {$a_3$};
\end{tikzpicture}
\caption{\ $|p_4p_5|=a$.}
\label{Fig:F32configb}
\end{subfigure}
\caption{Arrangement of points $p_4,p_5,a_1,a_2,a_3,a_4$.}
\end{figure}
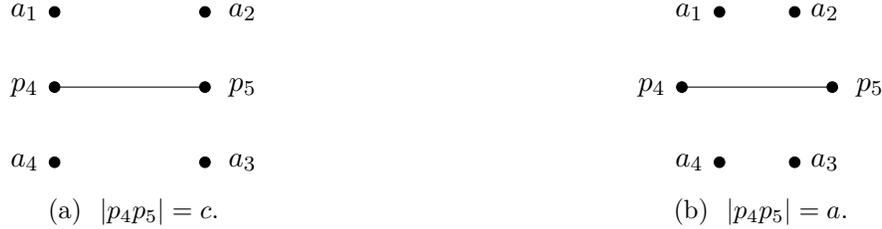
The triangles $a_1p_4p_5$, $a_2p_4p_5$ and $a_1a_2a_3$ are congruent to $T$ and thus $|a_1a_2|=|a_1p_4|$. This implies that the vertices $p_4,p_5,a_1,a_2,a_3,a_4$ form a regular hexagon. But then $T$ is of type $(90^\circ,60^\circ,30^\circ)$, a contradiction.

{\bf Case:} $T$ does not have a right angle.\\
Again, there are four points $a_1,a_2,a_3,a_4\in\mathbb{R}^2$ in the plane which together with $p_4p_5$ form a triangle congruent to $T$. Those four points form a rectangle. Since $145,245,345 \in E(F_{3,2})$, we have that $p_1,p_2,p_3\in \{a_1,a_2,a_3,a_4\}$. Since $123\in E(F_{3,2})$, the triangle $p_1p_2p_3$ is congruent to $T$. This is not possible since $T$ does not have a right angle. 

We conclude that $F_{3,2}$ is exactly forbidden for every triangle not of type $(90^\circ,60^\circ,30^\circ)$.
\end{proof}
\begin{lemma}\label{j4}
The hypergraph $J_4$ is forbidden for every triangle $T$. 
\end{lemma}
\begin{proof}
Recall that $E(J_4)=\{123,124,125,134,135,145\}$.  By Lemma~\ref{exactforbidden} it is sufficient to show that $J_4$ is exactly forbidden for $T$. Assume, for a contradiction, that $J_4$ is not exactly forbidden for $T$. Then there exist points $p_1,p_2,p_3,p_4,p_5\in \mathbb{R}^2$
 such that for every edge $xyz\in E(J_4)$, the triangle $p_xp_yp_z$ is congruent to $T$.

First, assume that the side lengths $a,b,c$ of $T$ are different.
 Since every pair $ij$ is contained in some edge in $J_4$, we have that  
the distances $|p_ip_j|\in \{a,b,c\}$ for $i\neq j\in[5]$. By the pigeonhole principle, and without loss of generality, there exists $i\neq j\in \{2,3,4,5\}$ such that $|p_ip_1|=|p_jp_1|=a$. Now, $p_1p_ip_j$ does not form a triangle congruent to $T$, a contradiction.

Now, assume that the side lengths of $T$ are $a,a,b\in \mathbb{R}$. If $a=b$, then all five points have pairwise distance $a$ from each other, which is not possible by Lemma~\ref{1distance}. Now, let $a\neq b$. Five points can only have pairwise distance $a$ or $b$ if they are arranged in a regular pentagon. The auxiliary hypergraph $\mathcal{H}(T,P)$ is either the empty graph or isomorphic to $C_5$ which does not contain a copy of $J_4$, a contradiction. We conclude that $J_4$ is exactly forbidden for $T$. 
\end{proof}

\begin{lemma}\label{k4}
Let $T$ be a triangle with no right angle, then $K_4^3$ is forbidden for $T$. 
\end{lemma}
\begin{proof}
Let $T$ be a triangle with side lengths $a,b,c$. By Lemma~\ref{exactforbidden} it suffices to show that $K_4^3$ is exactly forbidden for $T$. Assume, towards contradiction, that $K_4^3$ is not exactly forbidden for $T$. Then there exist points $p_1,p_2,p_3,p_4\in \mathbb{R}^2$
 such that each of the triangles $p_1p_2p_3,p_1p_2p_4,p_1p_3p_4,p_2p_3p_4$ is congruent to $T$. Note that $a=b=c$ is not possible, since by Lemma~\ref{1distance} there cannot be four points with pairwise the same distance. 

{\bf Case:} The side lengths $a,b,c$ are all different. \\
Each of the four triangles among $p_1,p_2,p_3,p_4$ contains each length exactly once. Thus, $|p_1p_2|=|p_3p_4|$, $|p_1p_3|=|p_2p_4|$ and $|p_1p_4|=|p_2p_3|$. Therefore, the convex hull of the four points is a parallelogram, whose diagonals have the same length, hence it is a rectangle. We conclude that each of the four triangles among the four points has a right angle, contradicting the assumption  that $T$ does not have a right angle.

{\bf Case:} Two of the side lengths $a,b,c$ are the same.\\
Say the lengths of the sides are $a,a,b$. Then among the four points $p_1,p_2,p_3,p_4$, the length $a$ has multiplicity four and $b$ has multiplicity two, the length $b$ sides forming a matching. Therefore the four points need to form a square, contradicting the assumption that $T$ does not have right angle.  

We conclude that $K_4^3$ is exactly forbidden for every triangle $T$ not having a right angle. 
\end{proof}

\begin{lemma}\label{k4minus}
Let $T$ be a triangle which has no right angle and is not of type $(120^\circ,3 0^\circ,$ $ 30^\circ)$. Then $K_4^{3-}$ is forbidden for $T$. 
\end{lemma}
\begin{proof}
Let $T$ be a triangle not of type $(120^\circ,30^\circ,30^\circ)$ with side lengths $a,b,c$. By Lemma~\ref{exactforbidden} it is sufficient to show that $K_4^{3-}$ is exactly forbidden for $T$. Assume, towards contradiction, that $K_4^{3-}$ is not exactly forbidden for $T$. Then there exist four points $p_1,p_2,p_3,p_4\in \mathbb{R}^2$
 such that the triangles $p_1p_2p_3$, $p_1p_2p_4$ and $p_1p_3p_4$ are congruent to $T$.
 Since every pair $ij$ is contained in some edge in $K_4^{3-}$, we have that the 
the distances $|p_i-p_j|\in \{a,b,c\}$ for $i\neq j\in[4]$. Note that $a=b=c$ is not possible, since by Lemma~\ref{1distance} there do not exist four points with the same pairwise distances between all pairs of  points.

{\bf Case:} All side lengths are different. \\
In this case, $|p_1p_2|,|p_1p_3|$ and $|p_1p_4|$ are all different, and therefore $p_2p_3p_4$ would also span a triangle congruent to $T$. This contradicts that $K_4^{3}$ is exactly forbidden by Lemmas~\ref{exactforbidden} and \ref{k4}.

{\bf Case:} Two of the side lengths are the same. \\
Assume that the side lengths of $T$ are $a,a,b$.
If $|p_1p_2|=|p_1p_3|=|p_1p_4|$, then $|p_2p_3|=|p_2p_4|=|p_3p_4|$ because $123,124,134\in E(K_4^{3-})$. In this case, $p_2p_3p_4$ spans an equilateral triangle with $p_1$ being its center. We conclude that $T$ is of type $(120^\circ,30^\circ,30^\circ)$, a contradiction. 
Otherwise, without loss of generality, $|p_1p_2|=|p_1p_3|=a$ and $|p_1p_4|=b$. Then, $|p_2p_3|=b$ and $|p_2p_4|=|p_3p_4|=a$. Therefore, the convex hull of the four points is a parallelogram, whose diagonals have the same length, hence it is a rectangle. Each of the four triangles among the four points has a right angle, contradicting that $T$ does not contain a right angle.

Therefore, $K_4^{3-}$ is exactly forbidden for $T$.
\end{proof}

\begin{lemma}
\label{C5}
Let $T$ be a triangle which is not of type $(108^\circ,36^\circ,36^\circ)$, $(72^\circ,72^\circ,36^\circ)$ or $(120^\circ,30^\circ,$ $30^\circ)$. Then $C_5$ is forbidden for $T$. 
\end{lemma}
\begin{proof}
Recall that $E(C_5)=\{123,234,345,451,512\}$.
Let $T$ be a triangle with side lengths $a,b,c$, which is not of type $(108^\circ,36^\circ,36^\circ)$, $(72^\circ,72^\circ,36^\circ)$ or $(120^\circ,30^\circ,30^\circ)$. Again, by Lemma~\ref{exactforbidden}, it is sufficient to show that $C_5$ is exactly forbidden for $T$.
Assume, towards contradiction, that $C_5$ is not exactly forbidden for $T$. Then there exist points $p_1,p_2,p_3,p_4,p_5\in \mathbb{R}^2$
 such that for every edge $xyz\in E(C_5)$, the triangle $p_xp_yp_z$ is congruent to $T$. Since every pair $ij$ is contained in some edge in $C_5$, we have that 
the distances $|p_ip_j|\in \{a,b,c\}$ for $i\neq j\in[5]$. Note that $a=b=c$ is not possible, since by Lemma~\ref{1distance} there cannot be five points with pairwise the same distance. 

Next, assume that two of the side lengths are the same. By Lemma~\ref{2distances}, the five points form a regular pentagon. We conclude that $T$ is of type $(108^\circ,36^\circ,36^\circ)$ or $(72^\circ,72^\circ,36^\circ)$, a contradiction.

Finally, assume that the three side lengths $a,b,c$ are different. Let $\alpha,\beta,\gamma$ be the angles of $T$, where $\alpha$ is opposite the side of length $a$, $\beta$ is opposite the side of length $b$ and $\gamma$ is opposite the side of length $c$. For some labeling of the side lengths, we can assume (recall the structure of $C_5$):
\begin{align*}
|p_1p_2|=a, \quad |p_2p_3|=b, \quad |p_3p_4|=a, \quad |p_4p_5|=b, \quad |p_1p_5|=c.
\end{align*}
This forces the rest of the distances:
\begin{align*}
|p_1p_3|=c, \quad |p_2p_4|=c, \quad |p_3p_5|=c,\quad |p_4p_1|=a, \quad |p_5p_2|=b.
\end{align*}
Now, the triangle $p_1p_3p_5$ is equilateral with side length $c$. See Figure~\ref{C5fig} for an illustration of the following argument.

\begin{figure}[h!]
\begin{center}
    \begin{tikzpicture}

\vertex[minimum size = 3pt, label={[xshift=-0.12cm, yshift=0cm]$p_{1}$}](0) at (0,0){};
\vertex[minimum size = 3pt, label=below:{$p_{3}$} ](1) at (2,0){};
\vertex[minimum size = 3pt, label=left:{$p_{5}$}](2) at (1,1.732) {};

\draw (0) -- (1);
\draw(0) -- (2);
\draw (2) -- (1);

\draw (1,1.732) -- ++(90:1.2cm);
\draw (1,1.732) -- ++(270:3cm);
\node at (1.3,-1) {$L_1$};
\draw (0,0) -- ++(30:3.8cm);
\draw (0,0) -- ++(210:1.4cm);
\node at (2.9,2) {$L_2$};
   \draw [domain=-60:90] plot ({0.8*cos(\x)+1}, {0.8*sin(\x)+1.732});
\node at (1.4,1.732) {\footnotesize $150^\circ$};

   \draw [domain=360:210] plot ({0.8*cos(\x)}, {0.8*sin(\x)});
\node at (0.1,-0.35) {\footnotesize $150^\circ$};

\end{tikzpicture}
\end{center}
\caption{Arrangement of points $p_1,p_2,p_3,p_4,p_5$ in the proof of Lemma~\ref{C5}.}
\label{C5fig}
\end{figure}
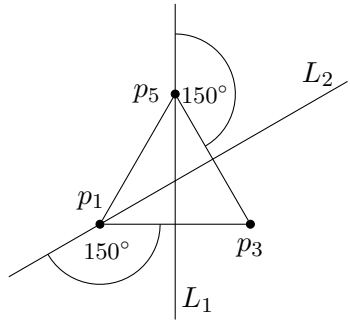

Since $|p_4p_1|=|p_4p_3|=a$, the point $p_4$ needs to be on the line $L_1$, which passes through $p_5$ and the midpoint of $p_1p_3$. This line $L_1$ intersects the line $p_3p_5$ at an angle of $30^\circ$. Since $p_3p_4p_5$ is congruent to $T$, we get $\alpha=30^\circ$ or $\alpha=150^\circ$. 

Similarly, since $|p_2p_3|=|p_2p_5|=b$, the point $p_2$ needs to be on the line $L_2$, which passes through $p_1$ and the midpoint of $p_3p_5$. This line $L_2$ intersects the line $p_1p_3$ at an angle of $30^\circ$. Since $p_1p_2p_3$ is congruent to $T$, we get $\beta=30^\circ$ or $\beta=150^\circ$.

Finally, because $\alpha+\beta+\gamma=180^\circ$, we conclude that $\alpha=\beta=30^\circ$ and $\gamma=120^\circ$, a contradiction. We conclude that $C_5$ is exactly forbidden for $T$. 
\end{proof}

\subsection{Some observations about small 3-graphs}
In this subsection we collect  several observations about $5$-vertex 3-graphs.
\begin{lemma}
\label{5vertexK_4^3-}
Let $H$ be a $K_4^{3-}$-free $3$-graph on $5$ vertices. If $H$ contains at least 3 edges, then $H$ is isomorphic to one of the following 3-graphs.
\begin{align*}
&E(C_5)=\{123,234,345,451,512\}, \quad \quad
&&E(C_5^-)=\{123,234,345,451\},\\
&E(F_{3,2})=\{123,124,125,345\}, \quad \quad
&&E(H_1)=\{123,124,135,145\},\\
&E(F_5)=\{123,124,345\}, \quad \quad
&&E(H_2)=\{123,124,135\}, \\
&E(H_3)=\{123,124,125\}. 
\end{align*}
\end{lemma}
\begin{proof}
First, let $H$ be a $K_4^{3-}$-free $3$-graph on $5$ vertices with exactly $3$ edges. Then $H$ contains $4$ vertices spanning exactly two edges. Without loss of generality, let $123,124\in E(H)$. Since $H$ is $K_4^{3-}$-free, hence $134,234\notin E(H)$. If $345\in E(H)$, then $H=F_5$. If $125\in E(H)$, then $H=H_3$. If $135\in E(H)$, $145\in E(H)$, $235\in E(H)$ or $245\in E(H)$, then $H$ is isomorphic to $H_2$. 

Next, let $H$ be a $K_4^{3-}$-free $3$-graph on $5$ vertices with exactly 4 edges. Then $H$ contains a copy of $F_5,H_2$ or $H_3$. If $H$ contains a copy of $F_5$, then $H$ is isomorphic to $F_{3,2}$ or $C_5^-$, because if $125$ is added to $F_5$, a copy of $F_{3,2}$ is created, and if $145,135,235$ or $245$ is added to $F_5$, a copy of $C_5^-$ is created. If $H$ contains a copy of $H_2$, then $H$ is isomorphic to $H_1$ or $C_5^-$, because if $145$ is added to $H_2$, $H_1$ is created, and if $235,245$ or $345$ is added to $H_2$, a copy of $C_5^-$ is created, and if $125$ is added to $H_2$ a copy of $K_4^{3-}$ is created. If $H$ contains a copy of $H_3$, then $H$ is isomorphic to $F_{3,2}$, because $345$ is the only edge which can be added to $H_3$ without creating a copy of $K_4^{3-}$.

There is no edge which can be added to $F_{3,2}$ or $H_1$ without creating a copy of $K_4^{3-}$. The only edge which can be added to $C_5^-$ without creating a copy of $K_4^{3-}$ is $512$, in which case a $C_5$ is created. Thus, up to isomorphism, $C_5$ is the unique $K_4^{3-}$-free $3$-graph on $5$ vertices with exactly $5$ edges. 

Trivially, there is no $K_4^{3-}$-free $3$-graph on $5$ vertices with more than 5 edges. 
\end{proof}

Given a 3-graph $H$, its \emph{shadow graph} $\delta H$ is defined to be the graph on vertex set $V(H)$ where a pair of vertices is an edge iff it if contained in an edge in $H$. We say that $H$ has \emph{complete shadow graph} if $\delta H$ is a clique on $|V(H)|$ vertices. The \emph{link graph} of a vertex $x$ in $H$ is the graph on vertex set $V(H)\setminus\{x\}$ where a pair of vertices $ab$ is an edge iff $abx\in E(H)$.

\begin{lemma}
\label{5vertexK_4^3-2}
Let $H$ be a $5$-vertex $K_4^{3-}$-free $3$-graph such that $\delta H$ is complete. Then $H$ is isomorphic to $C_5$ or $F_{3,2}$.
\end{lemma}
\begin{proof}
 Since $\binom{5}{2}=10$ and because $\delta H$ is complete, we conclude that $H$ contains at least $\left\lceil \frac{10}{3}\right\rceil =4$ edges. By Lemma~\ref{5vertexK_4^3-}, $H$ is isomorphic to $C_5, \ F_{3,2},\ C_5^-$ or $H_1$. The hypergraphs $C_5^-$ and $H_1$ do not have complete shadow graph. Thus, $H$ is isomorphic to $C_5$ or $F_{3,2}$. 
\end{proof}

\begin{lemma}
\label{5vertexK_4^3-C5F32}
There does not exist a $5$-vertex $\{K_4^{3-},C_5,F_{3,2}\}$-free 3-graph $H$ with complete shadow graph. 
\end{lemma}
\begin{proof}
 Let $H$ be a $\{K_4^{3-},C_5,F_{3,2}\}$-free 3-graph with 5 vertices such that $\delta H$ is complete. Because $\delta H$ is complete, the 3-graph $H$ contains at least $4$ edges. By Lemma~\ref{5vertexK_4^3-}, $H$ is isomorphic to $C_5^-$ or $H_1$. Neither of those hypergraphs has a complete shadow graph, a contradiction. 
\end{proof}

\begin{lemma}
\label{5vertexK_4^3}
Let $H$ be a $5$-vertex $\{C_5,J_4\}$-free $3$-graph, which contains a copy of $K_4^3$ and has complete shadow graph. Then $H$ is isomorphic to $H_4$, defined as
$$V(H_4)=[5], \quad E(H_4)=\{123,234,134,124,514,523\}.$$
\end{lemma}
\begin{proof}
Let $V(H)=[5]$ such that the vertices $1,2,3,4$ form a copy of $K_4^3$ in $H$. Since $H$ is $J_4$-free, the link graph of vertex $5$ does not contain a vertex of degree 3. Indeed, if some $a\in [4]$ has degree 3 in the link graph of vertex $5$, then $ab5\in E(H)$ for every $b\in [4]\setminus\{a\}$. However, since $1,2,3,4$ form a copy of $K_4^3$, we then have that $acd\in E(H)$ for every distinct $c,d\in [5]\setminus \{a\}$ and thus $H$ contains a copy of $J_4$, a contradiction.

Further, since $H$ is $C_5$-free, the link graph of vertex $5$ does not contain a path on 4 vertices. Indeed, if 5 contains a path $a,b,c,d$ in its link graph, then $5ab,5bc,5cd\in E(H)$. Since $1,2,3,4$ form a copy of $K_4^3$, we also have $bad,acd\in E(H)$. Therefore, $H$ contains a copy of $C_5$, a contradiction.

Because the shadow graph of $H$ is complete, every vertex of the link graph of $5$ has degree at least 1. Indeed, if $a\in [4]$ has degree 0 in the link graph of $5$, then $ab5\notin E(H)$ for every $b\in [4]\setminus \{a\}$. Then the pair $a5$ is not an edge in the shadow graph of $H$, a contradiction.

The only 4-vertex graph with minimum degree 1 and maximum degree 2, which does not contain a path on 4 vertices, is a matching of size 2. Therefore, we conclude that the link graph of $5$ is a matching of size $2$ and thus $H$ is isomorphic to $H_4$.
\end{proof}

\section{Lagrangian Method}
\label{lagrang}

Let $H$ be an $n$-vertex $3$-graph.
The \emph{Lagrangian polynomial} of $H$ is
\begin{align*}
\lambda_H(x_1,\ldots,x_{n}):= \sum_{ijk \in H} x_ix_jx_k,
\end{align*}
and the \emph{Lagrangian} of $H$ is
\begin{align*}
\lambda(H):= \max \{\lambda_H(x_1,\ldots,x_{n}): \ (x_1,x_2,\ldots,x_{n})\in \Delta_{k} \},
\end{align*}
where $\Delta_{n}=\{(x_1,x_2,\ldots,x_n)\in [0,1]^n: x_1+x_2+\ldots+x_n=1\}$ is the standard $(n-1)$-simplex in $\mathbb{R}^n$. The Lagrangian is an important function in determining $\pi(H)$, see \cite{MR771722}.

The following lemma reduces determining the function $h(n,T)$ to a weighted optimization problem with at most 20 variables.  
\begin{lemma}
 \label{lagrangian}
Let $T$ be a triangle and $n$ be a positive integer. Then there exists $\varepsilon_0>0$ such that for every $0<\varepsilon<\varepsilon_0$, there exists a point set $P$ of size $|P|\leq 20$ such that  $\mathcal{H}(T,P,\varepsilon)$ has complete shadow graph and 
\begin{align*}
h(n,T)\leq n^3 \lambda(\mathcal{H}(T,P,\varepsilon)).
\end{align*}
\end{lemma}
\begin{proof}
Let $T$ be a triangle with side lengths $a,b,c$. 
Let $\varepsilon_0>0$ be sufficiently small for the following argument to hold, in particular, such that for every $0<\varepsilon< \varepsilon_0$ there exists a point set $P=\{p_1,p_2,\ldots,p_n\}\subseteq \mathbb{R}^2$ of size $n$ such that $e(\mathcal{H}(T,P,\varepsilon))=h(n,T)$. For $\textbf{x'}=(\frac{1}{n},\ldots,\frac{1}{n})\in \Delta_n$, we have
\begin{align*}
        \frac{h(n,T)}{n^3}= \frac{e(\mathcal{H}(T,P,\varepsilon))}{n^3}= \lambda_{\mathcal{H}(T,P,\varepsilon)}(\textbf{x'}) \leq \lambda(\mathcal{H}(T,P,\varepsilon)).
\end{align*}
Let $\textbf{x}\in \Delta_n$ be such that $\lambda(\mathcal{H}(T,P,\varepsilon))=\lambda_{\mathcal{H}(T,P,\varepsilon)}(\textbf{x})$ with the fewest non-zero entries. Suppose that there exist two distinct points $p_i,p_j\in P$ such that the weights $x_i$ and $x_j$ are positive and there does not exist an edge in $\mathcal{H}(T,P,\epsilon)$ of the form $p_ip_jp_k$ with $x_k>0$.

Let $s_i$ be the sum of the products $x_{k} x_{k'}$ over all edges in $\mathcal{H}(T,P,\varepsilon)$ that are incident to $p_i$, i.e. they are of the form $p_ip_kp_{k'}$.
Define $s_j$ similarly for $p_j$, where we may
assume that $s_i \geq s_j$. For $k\in [n]$, define 
\begin{align*}
x^{*}_k=\begin{cases}
x_k & \text{if } k\neq i, k\neq  j, \\
x_i+x_j & \text{if } k=i, \\
0 & \text{if } k=j. 
\end{cases}
\end{align*}
Then,
\begin{align*}
\lambda_{\mathcal{H}(T,P,\varepsilon)}(\mathbf{x^{*}}) =  \lambda_{\mathcal{H}(T,P,\varepsilon)}(\mathbf{x})  + x_j s_i - x_j s_j \geq \lambda_{\mathcal{H}(T,P,\varepsilon)}(\mathbf{x}), 
\end{align*}
contradicting the choice of $\textbf{x}$. We conclude that for every two distinct points $p_i,p_j\in P$ with positive weights $x_i$ and $x_j$, there exists an edge in $\mathcal{H}(T,P,\epsilon)$ of the form $p_ip_jp_k$ with $x_k>0$. Let $P'\subseteq P$ be the set of points $p_i\in P$ with $x_i>0$. Then every point pair in $P'$ is contained in a triangle that is $\varepsilon$-congruent to $T$, 
\begin{align*}
h(n,T)\leq n^3\lambda_{\mathcal{H}(T,P,\varepsilon)}(\textbf{x})= n^3\sum_{ijk \in \mathcal{H}(T,P,\varepsilon)} x_ix_jx_k=n^3  \sum_{ijk \in \mathcal{H}(T,P',\varepsilon)} x_ix_jx_k
\end{align*}
and $\sum_{i: x_i>0}x_i=1$. Let $p_i,p_j\in P'$ be distinct. For any point $p_k\in P'\setminus\{p_i,p_j\}$, we have 
\begin{align*}
|p_kp_i|,|p_kp_j|\in (a-\varepsilon a,a+\varepsilon a) \cup (b-\varepsilon b,b+\varepsilon b) \cup (c-\varepsilon c,c+\varepsilon c). 
\end{align*}
There are at most 18 points in the plane which have one of the distances $a,b$ or $c$ from the points $p_i$ and $p_j$, because there are $3^2=9$ ways to choose the distances to $p_i$ and $p_j$, and for each such combination there are at most 2 points in the plane. 
Since every pair of points from $P'$ has distance at least $(1-2\varepsilon) \cdot \min\{a,b,c\}$ from each other, we conclude that $|P'\setminus\{p_i,p_j\}|\leq 18$ and thus $|P'|\leq 20$. The point set $P'$ has the desired properties.
\end{proof}

The following lemma is a refinement of Lemma~\ref{lagrangian}.  
\begin{lemma}
 \label{lagrangian2}
Let $T$ be a triangle and $n$ be a positive integer. Then there exists a point set $P$ of size $|P|\leq 7$ such that $\mathcal{H}(T,P)$ has complete shadow graph and 
\begin{align*}
h(n,T)\leq n^3 \lambda(\mathcal{H}(T,P)).
\end{align*}
\end{lemma}
\begin{proof}
Let $T$ be a triangle with side lengths $a,b,c$. Given a $3$-graph $H$ and $0\leq i \leq |V(H)|$, a sequence $(v^n)_n$ of vectors of $k$ points $v^n=(v^n_1,\ldots,v^n_k)\in (\mathbb{R}^2)^k$ is called $(H,i)$-\emph{partial}, if there exists a sequence of real numbers $\varepsilon_n \downarrow 0$ such that $v^n_xv^n_yv^n_z$ is $\varepsilon_n$-congruent to $T$ for $xyz\in E(H)$ and all $n\in  \mathbb{N}$, and further $v^n_j=v^{1}_j$ for all $n\in  \mathbb{N}$ and $j\leq i$.

By Lemma~\ref{lagrangian}, there exists $\varepsilon_0$ such that for all $0<\varepsilon<\varepsilon_0$ there exists a point set $P$ of size at most 20 such that every point pair in $P$ is contained in a triangle  $\varepsilon$-congruent  to $T$ and $h(n,T)\leq n^3 \lambda(\mathcal{H}(T,P,\varepsilon))$. Since the number of $3$-graphs on at most 20 vertices is finite, there exists a 3-graph $H$ with complete shadow graph on at most $20$ vertices such that $h(n,T)\leq n^3\lambda(H)$ and there exists an $(H,0)$-partial sequence. After shifting, we can assume that the first point in each vector is the origin, and thus we get an $(H,1)$-partial sequence.

Because $H$ has complete shadow graph, $12$ is contained in some edge $e\in E(H)$. After rotating around the origin, we can assume that the second point $v_n^2$ in each vector is on the $x$-axis with positive $x$-coordinate and has distance at most $2\varepsilon_n \cdot \min\{a,b,c\}$ from one of the points $(a,0),(b,0),(c,0)$. After replacing the second point in each vector with a point from $(a,0),(b,0),(c,0)$ which it is the closest to, we get an $(H,1)$-partial sequence, where the second point takes one out of three values. 
Going over to a subsequence where the second point is constant, we obtain an $(H,2)$-partial sequence. 

Now, assume that for some $2\leq i \leq |V(H)|$ there exists an $(H,i)$-partial sequence $(v^n)_n$. We will construct an $(H,i+1)$-partial sequence from it. 
 There are at most $18$ points $p\in \mathbb{R}^2$ such that $|v^1_1p|,|v^1_2p|\in \{a,b,c\}$. After replacing each $v^n_{i+1}$ with a point among those $18$ points it is the closest to, and going over to a subsequence where the second point is constant, we obtain an $(H,i+1)$-partial sequence. 

Iteratively, we obtain an $(H,|V(H)|)$-partial sequence, call it $(w^n)_n$. Then $w_j^n=w_j^1$ for all $n\in \mathbb{N}$ and $j\in [k]$. The point set $P:=\{w_1^1,\ldots,w_{|V(H)|}^1\}$ has the property that if $xyz\in E(H)$, then $w_x^1w_y^1w_z^1$ is $\varepsilon_n$-congruent to $T$ for all $n$ for some $\varepsilon_n\downarrow 0$, and thus also congruent to $T$. Therefore $\mathcal{H}(T,P)=H$ and
\begin{align*}
h(n,T)\leq n^3 \lambda(H)= n^3 \lambda(\mathcal{H}(T,P)).
\end{align*}
  Since the shadow graph of $H$ is complete, $P$ is a set in which pairwise distances take at most three different values, by Theorem~\ref{3distances}, we get $|P|\leq 7$. 
\end{proof}

\subsection{Application of the Lagrangian Method}

\begin{lemma}
\label{6gonlagrange}
Let $H_5$ be the $6$-vertex 3-graph with edges $142,143,145,146, 251,253,254,256,361,$ $362,364,365$. Then $\lambda(H_5)\leq \frac{1}{16}$.
\end{lemma}
\begin{proof}
We have
 \begin{align*}
 \lambda(H_5)&= \max_{\textbf{x}\in \Delta_{6}} \ x_1x_4(x_2+x_3+x_5+x_6)+ x_2x_5(x_1+x_3+x_4+x_6)+x_3x_6(x_1+x_2+x_4+x_5)\\
 &\leq \max_{\textbf{x}\in \Delta_{6}} \ \frac{(x_1+x_4)^2}{4}(x_2+x_3+x_5+x_6)+ \frac{(x_2+x_5)^2}{4}(x_1+x_3+x_4+x_6)\\
 &+
 \frac{(x_3+x_6)^2}{4}
 (x_1+x_2+x_4+x_5)\\
 &=
\frac{1}{4}\max_{\textbf{z}\in \Delta_{3}} \ z_1^2(z_2+z_3)+z_2^2(z_1+z_3)+z_3^2(z_1+z_2)\\
&= 
\frac{1}{4} \max_{\textbf{z}\in \Delta_{3}} \ z_1z_2(z_1+z_2)+z_1z_3(z_1+z_3)+z_2z_3(z_2+z_3) \\
&= 
\frac{1}{4} \max_{\textbf{z}\in \Delta_{3}} \ z_1z_2+z_1z_3+z_2z_3-3z_1z_2z_3 \\
&=
\frac{1}{4}\ \max_{\textbf{z}\in \Delta_{3}, \ z_3\leq \frac{1}{3}} \ z_1z_2(1-3z_3)+(z_1+z_2)z_3 \\
&\leq \frac{1}{4}\max_{\textbf{a}\in \Delta_{2}} \ \frac{a_1}{4}(1-3a_2)+a_1a_2\leq \frac{1}{4}\max_{0\leq a \leq 1} \ \frac{a}{4}(3a-2)+a(1-a)
\leq \frac{1}{16},\qedhere
 \end{align*}
 where in the first inequality we used the AM-GM inequality, and in the first equality the substitutions are $z_1:=x_1+x_4,\ z_2:=x_2+x_5$ and $z_3:=x_3+x_6$.
\end{proof}

\begin{lemma}\label{306090}
Let $T$ be a triangle of type $(90^\circ,60^\circ,30^\circ)$. Then $h(n,T)\leq \frac{n^3}{16}$.
\end{lemma}
\begin{proof}

Without loss of generality let the side lengths of $T$ be $1,\sqrt{3}$ and $2$. By Lemma~\ref{lagrangian2}, there exists a point set $P$ of size at most $7$ such that $\mathcal{H}(T,P)$ has complete shadow graph and $h(n,T)\leq n^3 \lambda({\mathcal{H}(T,P)})$.

If $|P|\leq 4$, then $\lambda({\mathcal{H}(T,P)})\leq \lambda(K_4^{3})=\frac{1}{16}$. Next, assume $|P|=5$. If ${\mathcal{H}(T,P)}$ is $K_4^{3-}$-free, then by Lemma~\ref{5vertexK_4^3-2}, the hypergraph ${\mathcal{H}(T,P)}$ is isomorphic to $C_5$ or $F_{3,2}$. However, $\lambda(C_5)=\frac{1}{25}$ and $\lambda(F_{3,2})=\frac{189+15\sqrt{5}}{6 \cdot  961}\leq 0.04$ as observed in \cite{BaberTalbot}. In particular, $\lambda({\mathcal{H}(T,P)})\leq \frac{1}{16}$. Therefore, we can assume that ${\mathcal{H}(T,P)}$ contains a copy of $K_4^{3-}$. The corresponding four points form a rectangle with side lengths $1$ and $\sqrt{3}$, and therefore ${\mathcal{H}(T,P)}$ contains in fact a copy of $K_4^3$. By Lemmas~\ref{C5} and \ref{j4}, ${\mathcal{H}(T,P)}$ is $C_5$ and $J_4$-free. Therefore, by Lemma~\ref{5vertexK_4^3}, $\mathcal{H}(T,P)$ is isomorphic to the $3$-graph $H_4$. The $3$-graph $H_4$ is contained as a copy in the $3$-graph $H_5$, which can be seen by observing that in $H_5$ the vertices $1,2,4,5$ span a copy of  $K_4^3$ and the link graph of vertex $3$ on this set is a matching. Then, by Lemma~\ref{6gonlagrange}, $$\lambda(\mathcal{H}(T,P))= \lambda(H_4)\leq \lambda(H_5) \leq \frac{1}{16}.$$

If $|P|\geq 6$, then, by Theorems~\ref{3distances} and \ref{3distances6}, $P$ is contained in a point set forming a regular 6-gon with its center, or $P$ forms an equilateral triangle of side length 2 together with the midpoints of the three sides. It is not possible that $P$ forms an equilateral triangle of side length 2 together with the midpoints of the three sides, because then two of the points were not contained in a triangle congruent to $T$, as can be seen in Figure~\ref{Fig:shapes3} (a).

Thus $P$ is contained in a point set forming a regular 6-gon with its center. Since the center is not contained in any triangle  congruent to $T$, we conclude that $P$ is a regular $6$-gon. Therefore, $\mathcal{H}(T,P)$ is isomorphic to $H_5$, as can be seen in Figure~\ref{Fig:6gon} and thus $\lambda(\mathcal{H}(T,P))=\lambda(H_5)\leq \frac{1}{16}$ by Lemma~\ref{6gonlagrange}. 
\begin{figure}[h!]
\begin{subfigure}{.5\textwidth}
\centering
\begin{tikzpicture}

\draw
\foreach \i in {1,...,6}{
(90+360/6*\i:0.8) coordinate(\i) node[vtx]{}
};

\draw(1) -- (2);
\draw (2) -- (3);
\draw (3) -- (4);
\draw (4) -- (5);
\draw (5) -- (6);
\draw (6) -- (1);

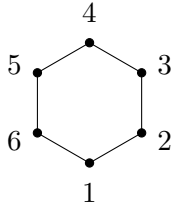
\node[] at (0,-1.2) {$1$};
\node[] at (0,1.2) {$4$};
\node[] at (1,0.5) {$3$};
\node[] at (-1,0.5) {$5$};
\node[] at (1,-0.5) {$2$};
\node[] at (-1,-0.5) {$6$};

\end{tikzpicture}
\end{subfigure}
\begin{subfigure}{.5\textwidth}
\centering

The triangles congruent to $T$ are $142,143,145,$ $146, 251,253,254,256,361,362,364,365$.
\vspace{1cm}

\end{subfigure}
\caption{A regular $6$-gon.}
\label{Fig:6gon}
\end{figure}

We conclude $h(n,T)\leq n^3\lambda(\mathcal{H}(T,P))\leq \frac{n^3}{16}$. 
\end{proof}

\begin{lemma}\label{30-30-120}
Let $T$ be a triangle of type $(120^\circ,30^\circ,30^\circ)$. Then $h(n,T)\leq \frac{4}{81}n^3$.
\end{lemma}
\begin{proof}
Let $T$ be a triangle of type $(120^\circ,30^\circ,30^\circ)$. By Lemma~\ref{lagrangian2} there exists a point set $P$ of size $|P|\leq 7$ such that $\mathcal{H}(T,P)$ has complete shadow graph and
$h(n,T)\leq n^3 \lambda({\mathcal{H}(T,P)})$.

Since every pair is contained in a  triangle congruent to $T$, the distance between every  pairs of points in $P$ is one of two values. By Lemma~\ref{2distances}, we conclude $|P|\leq 5$. If $|P|=3$, then $\lambda({\mathcal{H}(T,P)})\leq \max_{x\in \Delta_3} x_1x_2x_3\leq \frac{1}{27}$. If $|P|=4$, then by Lemma~\ref{k4}, $\mathcal{H}(T,P)$ is $K_4^{3}$-free. Thus, 
\begin{align*}
\lambda({\mathcal{H}(T,P)})\leq \max_{\textbf{x}\in \Delta_4} x_1x_2x_3+x_1x_2x_4+x_1x_3x_4=\lambda(K_4^{3-})= \frac{4}{81},
\end{align*}
where the maximum is obtained when $x_1=\frac{1}{3}$ and $x_2=x_3=x_4=\frac{2}{9}$, as e.g. observed in \cite{BaberTalbot}. 
If $|P|=5$, then by Lemma~\ref{2distances}, the point set $P$ forms a regular pentagon. We conclude that $\mathcal{H}(T,P)$ is the empty graph, contradicting that every pair of points in $P$ is contained in a  triangle congruent to $T$. We conclude $h(n,T)\leq \frac{4}{81}n^3$.
\end{proof}

\begin{lemma}\label{3636108}
Let $T$ be a triangle of type $(108^\circ,36^\circ,36^\circ)$ or $(72^\circ,72^\circ,36^\circ)$. Then $h(n,T)\leq \frac{n^3}{25}$.
\end{lemma}
\begin{proof}
Let $T$ be a triangle of type $(36^\circ,36^\circ,108^\circ)$ or $(72^\circ,72^\circ,36^\circ)$. By Lemma~\ref{lagrangian2} there exists a point set $P$ of size at most $7$ such that $\mathcal{H}(T,P)$ has complete shadow graph and $h(n,T)\leq n^3 \lambda({\mathcal{H}(T,P)})$.

Since every pair is contained in a  triangle congruent to $T$, the distance between every  pair of points in $P$ is one of two values. By Lemma~\ref{2distances}, we conclude $|P|\leq 5$. 

If $|P|=3$, then $\lambda({\mathcal{H}(T,P)})\leq \max_{\textbf{x}\in \Delta_3} x_1x_2x_3\leq \frac{1}{27}$. If $|P|=4$, then by Lemma~\ref{k4minus}, $\mathcal{H}(T,P)$ is $K_4^{3-}$-free, but then there are at most two edges in $\mathcal{H}(T,P)$, contradicting that every pair of points in $P$ is contained in a congruent triangle to $T$. 

If $|P|=5$, then by Lemma~\ref{2distances}, the point set $P$ forms a regular pentagon. We conclude that $\mathcal{H}(T,P)$ is either the empty graph or is isomorphic to $C_5$. Since every pair of points in $P$ is contained in a congruent triangle to $T$, the $3$-graph $\mathcal{H}(T,P)$ cannot be the empty graph. We conclude that $\mathcal{H}(T,P)$ is isomorphic to $C_5$ and obtain 
\begin{align*}
\lambda({\mathcal{H}(T,P)})=\lambda(C_5)=  \max_{\textbf{x}\in \Delta_{5}}\ x_1x_2x_3+x_2x_3x_4+x_3x_4x_5+x_4x_5x_1+x_5x_1x_2=\frac{1}{25},
\end{align*}
where the maximum is obtained at $x_1=x_2=x_3=x_4=x_5=\frac{1}{5}$ as e.g. observed in \cite{BaberTalbot}.
\end{proof}

\begin{lemma}\label{mosttriangles}
Let $T$ be not right angled, and not of type $(120^\circ,30^\circ,30^\circ)$, $(\frac{4\cdot 180}{7}^\circ,\frac{2\cdot 180}{7}^\circ,\frac{180}{7}^\circ)$, $(108^\circ,36^\circ,36^\circ)$ or $(72^\circ,72^\circ,36^\circ)$. Then, $h(n,T)\leq \frac{n^3}{27}$.
\end{lemma}
\begin{proof}
Without loss of generality, let the side lengths of $T$ be $1\leq b \leq c$.
Again, by Lemma~\ref{lagrangian2} there exists a point set $P$ of size at most $7$ such that $\mathcal{H}(T,P)$ has complete shadow graph and $h(n,T)\leq n^3 \lambda({\mathcal{H}(T,P)})$.

By Lemmas~\ref{f32}, \ref{k4minus}, \ref{C5} the hypergraph $\mathcal{H}(T,P)$ is $\{F_{3,2},K_4^{3-},C_5\}$-free.  
Since every pair is contained in a triangle congruent to $T$, the distances between two points of $P$ take values $1,b$ or $c$.

If $|P|=3$, then trivially $\lambda({\mathcal{H}(T,P)}\leq \frac{1}{27}$. If $|P|=4$, because $\mathcal{H}(T,P)$ is $K_4^{3-}$-free, there are at most two edges in $\mathcal{H}(T,P)$, contradicting that every pair of points in $P$ is contained in a triangle congruent to $T$. By Lemma~\ref{5vertexK_4^3-C5F32}, $|P|=5$ is not possible. If $|\{1,b,c\}|\leq 2$, then $|P|\leq 5$ by Lemma~\ref{2distances} and  thus $\lambda({\mathcal{H}(T,P)}\leq \frac{1}{27}$ by the  previous analysis. Therefore, we can assume that $|\{1,b,c\}|=3$.

If $|P|=6$, then by Lemma~\ref{2distances}, $P$ is a $3$-distance set. If $P$ is not a maximal 3-distance set, then by Theorem~\ref{3distances}, $P$ is contained in a regular 7-gon or a regular $6$-gon with its center. Again, this implies that $T$ is of type $(\frac{4\cdot 180}{7}^\circ,\frac{2\cdot 180}{7}^\circ,\frac{180}{7}^\circ)$ or it is right angled, a contradiction. If $P$ is a maximal 3-distance set, then by Theorem~\ref{3distances6}, $P$ is isomorphic to one of the point sets depicted in Figure~\ref{Fig:shapes3} (a)-(f). However, in each of those point sets there exists a pair of vertices which is not contained in a triangle congruent to $T$, a contradiction.

If $|P|=7$, then by Theorem~\ref{3distances}, $P$ forms a regular 7-gon or a regular $6$-gon with its center. Thus, $T$ is of type $(\frac{4\cdot 180}{7}^\circ,\frac{2\cdot 180}{7}^\circ,\frac{180}{7}^\circ)$ or it is right angled, a contradiction.

\end{proof}

\section{Proof of Theorem~\ref{upperbounds}}
\label{final}
The lower bounds on $h(n,T)$ have been presented in the introduction. Here, we present the corres\-ponding upper bounds. The following lemma translates asymptotic bounds into exact bounds.

\begin{lemma}
\label{asymp}
Let $T$ be a triangle. If $h(n,T)\leq a n^3(1+o(1))$ for some $a\in [0,1]$, then $h(n,T)\leq an^3$ for every positive integer $n$.
\end{lemma}
\begin{proof}
Assume, for a contradiction, that there exists a positive integer $k$ such that $h(k,T)>ak^3$. Let $\varepsilon>0$. There exists a point set $P$ with $k$ points such that the number of triangles that are $\frac{\varepsilon}{2}$-congruent to $T$ is at least $\lfloor ak^3 \rfloor+1$. Now, we construct a new point set $P'$ from $P$ with $nk$  points, where for each point $p\in P$ we add $n-1$ new points in the disk of radius $\frac{\varepsilon}{4}$ centered around $p$. The number of triangles in $P'$ that are $\varepsilon$-congruent to $T$ is at least $(\lfloor ak^3 \rfloor+1) n^3$. Since $\varepsilon>0$ was arbitrary small, we get
\begin{align*}
h(nk,T) \geq (\lfloor ak^3 \rfloor+1) n^3= \frac{\lfloor ak^3 \rfloor+1}{k^3} (nk)^3,
\end{align*}
contradicting $h(n,T)\leq a n^3(1+o(1))$, because $\frac{\lfloor ak^3 \rfloor+1}{k^3}>a$.
\end{proof}

\begin{proof}[Proof of Theorem~\ref{upperbounds}]
The upper bounds for the statements (b), (d) and (e) hold by Lemmas~\ref{30-30-120},~\ref{3636108} and \ref{mosttriangles}, respectively.\\

\noindent (a): Let $T$ be a triangle not of type $(90^\circ,60^\circ,30^\circ)$, then by Lemmas~\ref{f32} and \ref{j4}, the hypergraphs $F_{3,2}$ $J_4$ are forbidden for $T$. We conclude
\begin{align*}
h(n,T)\leq \textup{ex}(n,\{F_{3,2},J_4\})\leq \frac{n^3}{16} (1+o(1)),
\end{align*}
where the last inequality holds because $\pi(F_{3,2},J_4)=\frac{3}{8}$, which was proved by Falgas-Ravry and Vaughan~\cite{RavryTuran} using Razborov's flag algebra method\cite{RazbarovK43,flagsRaz}. By Lemma~\ref{asymp}, we conclude $h(n,T)\leq \frac{n^3}{16}$. If $T$ is of type $(90^\circ,60^\circ,30^\circ)$, then by Lemma~\ref{306090}, $h(n,T)\leq \frac{n^3}{16}$.\\

\noindent (c): In this case we know that by Lemmas~\ref{f32}, \ref{k4minus} and \ref{C5}, the hypergraphs $F_{3,2},K_4^{3-}$ and $C_5$ are forbidden for $T$. We conclude
\begin{align*}
h(n,T)\leq \textup{ex}(n,\{K_4^{3-},F_{3,2},C_5\})\leq \frac{2}{49}n^3 (1+o(1)),
\end{align*}
where the last inequality holds because $\pi(\{K_4^{3-},F_{3,2},C_5\})=\frac{12}{49}$, which was proved by Falgas-Ravry and Vaughan~\cite{RavryTuran} using Razborov's flag algebra method. By Lemma~\ref{asymp}, we conclude $h(n,T)\leq \frac{2}{49}n^3$. 
\end{proof}

\bibliographystyle{abbrvurl}
\bibliography{final}

\end{document}